\documentclass[10pt]{amsart}
\usepackage{enumerate,amsmath,amssymb,latexsym,
amsfonts, amsthm, amscd, mathabx}

\theoremstyle{plain}

\newtheorem{theorem}{Theorem}

\numberwithin{equation}{section}

\begin{document}

\title {Marcel Riesz on N\"orlund Means}

\date{}

\author[P.L. Robinson]{P.L. Robinson}

\address{Department of Mathematics \\ University of Florida \\ Gainesville FL 32611  USA }

\email[]{paulr@ufl.edu}

 \keywords{}

\begin{abstract}

We note that the necessary and sufficient conditions established by Marcel Riesz for the inclusion of {\it regular} N\"orlund summation methods are in fact applicable quite generally. 

\end{abstract}

\maketitle

\section*{Introduction} 

\medbreak 

One of the simplest classes of summation methods for divergent series was introduced independently by N\"orlund [2] in 1920 and by Voronoi in 1901 with an annotated English translation [5] by Tamarkin in 1932. Explicitly, let $(p_n : n \geqslant 0)$ be a real sequence, with $p_0 > 0$ and with $p_n \geqslant 0$ whenever $n > 0$; when $n \geqslant 0$ let us write $P_n = p_0 + \cdots + p_n.$ To each sequence $s = (s_n : n \geqslant 0)$ is associated the sequence $N^p s$ of {\it N\"orlund means} defined by 
$$ (N^ps)_m = \frac{p_0 s_m + \cdots + p_m s_0}{p_0 + \cdots + p_m} = \frac{1}{P_m} \sum_{n = 0}^m p_{m - n} s_n.$$
We say that the sequence $s$ is $(N, p)$-{\it convergent} to $\sigma$ precisely when the sequence $N^p s$ converges to $\sigma$ in the ordinary sense, writing this as 
$$s \xrightarrow{(N, p)} \sigma$$
or as $s \rightarrow \sigma \; (N, p)$; viewing the formation of N\"orlund means as a summation method, when $(s_n : n \geqslant 0)$ happens to be the sequence of partial sums of the series $\sum_{n \geqslant 0} a_n$ we may instead say that this series is $(N, p)$-{\it summable} to $\sigma$ and write 
$$\sum_{n = 0}^{\infty} a_n = \sigma \: (N, p).$$

\medbreak 

An important question regarding N\"orlund summation methods (and summation methods in general) concerns inclusion. We say that $(N, q)$ {\it includes} $(N, p)$ precisely when each $(N, p)$-convergent sequence is $(N, q)$-convergent to the same limit; equivalently, when each $(N, p)$-summable series is $(N, q)$-summable to the same sum. This relationship will be symbolized by $(N, p) \rightsquigarrow (N, q)$. The important notion of regularity may be seen as a special case of inclusion: the N\"orlund method $(N, q)$ is said to be {\it regular} precisely when each ordinarily convergent sequence is $(N, q)$-convergent to the same limit; that is, precisely when $(N, u) \rightsquigarrow (N, q)$ where $u_0 = 1$ and where $u_n = 0$ whenever $n > 0$. Precise necessary and sufficient conditions for the {\it regular} N\"orlund method $(N, q)$ to include the {\it regular} N\"orlund method $(N, p)$ were determined by Marcel Riesz and communicated to Hardy in a letter, an extract from which appeared as [3]. The line of argument indicated by Riesz in his letter was amplified by Hardy in his classic treatise `{\it Divergent Series}' [1], which we recommend for further information regarding summation methods in general and N\"orlund methods in particular. 

\medbreak 

Our primary purpose here is to point out that the necessary and sufficient `Riesz' conditions in fact apply to N\"orlund methods quite generally, without regularity hypotheses. 

\medbreak 

\section*{Inclusive Riesz Conditions}

\medbreak 

A celebrated theorem of Silverman, Steinmetz and Toeplitz gives necessary and sufficient conditions for a linear summation method to be regular, and proves to be very useful. The infinite matrix $[c_{m, n} : m, n \geqslant 0]$ yields a linear summation method $C$ by associating to each sequence $s = (s_n : n \geqslant 0)$ a corresponding sequence $t = (t_m : m \geqslant 0)$ given by 
$$t_m := \sum_{n = 0}^{\infty} c_{m, n} s_n$$
assumed convergent; to say that this linear summation method is {\it regular} is to say that, whenever the sequence $s$ is convergent, the sequence $t$ is convergent and $\lim_{m \rightarrow \infty} t_m = \lim_{n \rightarrow \infty} s_n$. The Silverman-Steinmetz-Toeplitz theorem may now be stated as follows. 

\medbreak 

\begin{theorem} \label{SST} 
The linear summation method $C$ with matrix $[c_{m, n} : m, n \geqslant 0]$ is regular precisely when each of the following conditions is satisfied: \\
{\rm (i)} there exists $H \geqslant 0$ such that for each $m \geqslant 0$ 
$$\sum_{n = 0}^{\infty} |c_{m, n}| \leqslant H;$$\\
{\rm (ii)} for each $n \geqslant 0$ 
$$\lim_{m \rightarrow \infty} c_{m, n} = 0;$$\\
{\rm (iii)} $$\lim_{m \rightarrow \infty} \sum_{n = 0}^{\infty} c_{m, n} = 1.$$
\end{theorem} 

\begin{proof} 
This appears conveniently as Theorem 2 in [1]. 
\end{proof}

\medbreak 

Now, let $(N, p)$ and $(N, q)$ be N\"orlund summation methods, or N\"orlunds for short. As $p_0$ is nonzero, the (triangular Toeplitz) system 
$$q_n = k_0 p_n + \cdots + k_n p_0 \; \; \; (n \geqslant 0)$$
is solved (recursively) by a unique sequence $k = (k_n : n \geqslant 0)$ of comparison coefficients; by summation, it follows that whenever $n \geqslant 0$ also 
$$Q_n = k_0 P_n + \cdots + k_n P_0.$$
The comparison sequence $k$ generates a (formal) power series 
$$k(x) = \sum_{n \geqslant 0} k_n x^n$$
while the N\"orlund sequences $p$ and $q$ also generate their own power series; the convolution relation $q = k * p$ between sequences corresponds to the relation 
$$q(x) = k(x) p(x)$$
between generating functions. We remark that if the N\"orlunds $(N, p)$ and $(N, q)$ are regular, their power series $p(x)$ and $q(x)$ converge whenever $|x| < 1$; the nonvanishing of $p(0) = p_0$ then ensures that the power series $k(x)$ converges when $|x|$ is small. 

\medbreak 

The introduction of the sequence $(k_n : n \geqslant 0)$ of comparison coefficients facilitates the following convenient expression for the N\"orlund means determined by $(N, q)$ in terms of the N\"orlund means determined by $(N, p)$. 

\medbreak 

\begin{theorem} \label{NN}
If $r = (r_n : n \geqslant 0)$ is any sequence then 
$$ (N^q r)_m = \sum_{n = 0}^{\infty} c_{m, n} (N^p r)_n$$
where if $n > m$ then $c_{m, n} = 0$ while if $n \leqslant m$ then $c_{m, n} = k_{m - n} P_n / Q_m$. 
\end{theorem} 

\begin{proof} 
Direct calculation: simply take the definition 
$$Q_m (N^q r)_m = q_0 r_m + \cdots + q_m r_0$$
and rearrange thus 
$$k_0 p_0 r_m + \cdots + (k_0 p_m + \cdots + k_m p_0) r_0 = k_0 (p_0 r_m + \cdots + p_m r_0) + \cdots + k_m (p_0 r_0)$$
to obtain 
$$Q_m (N^q r)_m = k_0 P_m (N^p r)_m + \cdots + k_m P_0 (N^p r)_0.$$
\end{proof} 

\medbreak 

We note that this result appears in the proof of [1] Theorem 19 but is there recorded only for regular N\"orlunds and established by comparing power series expansions; the argument presented here (essentially due to N\"orlund) is taken from [1] Theorem 17 and comes directly from the comparison coefficients without involving regularity.  

\medbreak 

The Riesz conditions ${\bf R}_{p q}$ associated to the N\"orlunds $(N, p)$ and $(N, q)$ may now be stated as follows: 
\medbreak 
${\bf R}_{p q}^1$: there exists $H \geqslant 0$ such that for each $m \geqslant 0$ 
$$|k_0| P_m + \cdots + |k_m| P_0 \leqslant H Q_m;$$ 
\medbreak 
${\bf R}_{p q}^2$: the sequence $(k_m / Q_m : m \geqslant 0)$ converges to zero. 
\medbreak 
As mentioned in the introduction, the fact that ${\bf R}_{p q}^1$ and ${\bf R}_{p q}^2$ are both necessary and sufficient for the inclusion $(N, p) \rightsquigarrow (N, q)$ between {\it regular} N\"orlunds appeared in [3] and was elaborated in [1] where it becomes Theorem 19. In what follows, we re-examine the line of argument taken in [3] and [1], deliberately stripping regularity hypotheses. 

\medbreak 

Henceforth, we shall write $C_{p q}$ for the linear summation method with matrix $[c_{m, n} : m, n \geqslant 0]$ expressing $(N, q)$ in terms of $(N, p)$ as in Theorem \ref{NN}. 

\medbreak 

On the one hand, we relate inclusion $(N, p) \rightsquigarrow (N, q)$ to regularity of $C_{p q}$. 

\medbreak 

\begin{theorem} \label{inclusionreg}
The inclusion $(N, p) \rightsquigarrow (N, q)$ holds precisely when the linear summation method $C_{p q}$ is regular. 
\end{theorem} 

\begin{proof} 
Assume $(N, p) \rightsquigarrow (N, q)$. Let $s = (s_n : n \geqslant 0)$ be any sequence. Note that $s = N^p r$ for a unique sequence $r = (r_n : n \geqslant 0)$ found by recursively solving the triangular Toeplitz system 
$$P_n s_n = p_0 r_n + \cdots + p_n r_0 \; \; \; (n \geqslant 0).$$
According to Theorem \ref{NN}, if $m \geqslant 0$ then 
$$t_m := \sum_{n = 0}^{\infty} c_{m, n} s_n = \sum_{n = 0}^{\infty} c_{m, n} (N^p r)_n = (N^q r)_m.$$
Now, let $s \rightarrow \sigma$: then $N^p r \rightarrow \sigma$ (by choice of $r$) hence $r \xrightarrow{(N, p)} \sigma$ (by definition of $(N, p)$-convergence) so that $r \xrightarrow{(N, q)} \sigma$ (by the $(N, p) \rightsquigarrow (N, q)$ assumption) whence $N^q r \rightarrow \sigma$ (by definition of $(N, q)$-convegence); that is, $t \rightarrow \sigma$. This proves that $C_{p q}$ is regular.  
\medbreak 
Assume that $C_{p q}$ is regular. Let $r = (r_n : n \geqslant 0)$ be $(N, p)$-convergent to $\sigma$: then $N^p r \rightarrow \sigma$ so Theorem \ref{NN} and the regularity of $C_{p q}$ yield $N^q r \rightarrow \sigma$; thus,  $r$ is $(N, q)$-convergent to $\sigma$ also. This proves $(N, p) \rightsquigarrow (N, q)$. 
\end{proof} 

\medbreak 

On the other hand, we relate regularity of $C_{p q}$ to the Riesz conditions ${\bf R}_{p q}$.  

\medbreak 

\begin{theorem} \label{regRiesz}
The linear summation method $C_{p q}$ is regular precisely when the Riesz conditions ${\bf R}_{p q}^1$ and ${\bf R}_{p q}^2$ are satisfied. 
\end{theorem} 

\begin{proof} 
Assume $C_{p q}$ to be regular and invoke Theorem \ref{SST}. Part (i) furnishes $H \geqslant 0$ such that for each $m \geqslant 0$ 
$$\frac{|k_m| P_0 + \cdots + |k_0| P_m}{Q_m} =  \sum_{n = 0}^{m} \Big|\frac{k_{m - n} P_n}{Q_m}\Big| = \sum_{n = 0}^{\infty} |c_{m, n}| \leqslant H$$
whence ${\bf R}_{p q}^1$ holds. Part (ii) says that $\lim_{m \rightarrow \infty} c_{m, n} = 0$ for each $n \geqslant 0$; in particular, 
$$0 = \lim_{m \rightarrow \infty} c_{m, 0} = \lim_{m \rightarrow \infty} \frac{k_m}{Q_m} P_0$$
whence ${\bf R}_{p q}^2$ holds. 
\medbreak 
Assume that ${\bf R}_{p q}^1$ and ${\bf R}_{p q}^2$ are satisfied. Theorem \ref{SST}(i)  holds because 
$$\sum_{n = 0}^{\infty} |c_{m, n}| = \sum_{n = 0}^{m} \Big|\frac{k_{m - n} P_n}{Q_m}\Big| = \frac{|k_m| P_0 + \cdots + |k_0| P_m}{Q_m} \leqslant H$$
on account of ${\bf R}_{p q}^1$. Theorem \ref{SST}(ii) holds because 
$$|c_{m, n}| = \Big|\frac{k_{m - n} P_n}{Q_m}\Big| \leqslant \frac{|k_{m - n}|}{Q_{m - n}} P_n  \rightarrow 0 \; \; {\rm as} \; \; m \rightarrow \infty$$
on account of ${\bf R}_{p q}^2$. Finally, Theorem \ref{SST}(iii) holds simply because of the relation
$$Q_m = k_0 P_m + \cdots + k_m P_0.$$ Theorem \ref{SST} now guarantees that $C_{p q}$ is regular. 
\end{proof} 

\medbreak 

In conclusion, the Riesz conditions ${\bf R}_{p q}$ are both necessary and sufficient for the inclusion $(N, p) \rightsquigarrow (N, q)$ without any assumptions of regularity. 

\medbreak 

\begin{theorem} 
Let $(N, p)$ and $(N, q)$ be any N\"orlund methods. The inclusion $(N, p) \rightsquigarrow (N, q)$ holds precisely when the Riesz conditions ${\bf R}_{p q}^1$ and ${\bf R}_{p q}^2$ are satisfied. 
\end{theorem} 

\begin{proof} 
Simply combine Theorem \ref{inclusionreg} and Theorem \ref{regRiesz}. 
\end{proof}

\bigbreak 

\begin{center} 
{\small R}{\footnotesize EFERENCES}
\end{center} 
\medbreak 

[1] G.H. Hardy, {\it Divergent Series}, Clarendon Press, Oxford (1949). 

\medbreak 

[2] N.E. Norlund, {\it Sur une application des fonctions permutables}, Lunds Universitets \r{A}rsskrift (2) Volume 16  Number 3 (1920) 1-10. 

\medbreak 

[3] M. Riesz, {\it Sur l'\'equivalence de certaines m\'ethodes de sommation}, Proceedings of the London Mathematical Society (2) {\bf 22} (1924) 412-419. 

\medbreak 

[4] P.L. Robinson, {\it Finite N\"orlund Summation Methods}, arXiv 1712.06744 (2017).

\medbreak 

[5] G.F. Voronoi, {\it Extension of the Notion of the Limit of the Sum of Terms of An Infinite Series}, Annals of Mathematics (2) Volume 33 Number 3 (1932) 422-428. 

\end{document}